\documentclass[10pt]{amsart}
\usepackage{amsmath}
\usepackage{amssymb}
\usepackage{enumerate}
\usepackage{enumitem}
\usepackage{anysize}
\marginsize{1.25cm}{1.25cm}{1.75cm}{2.5cm}

\theoremstyle{plain}
\newtheorem{theor}{Theorem}[section]
{}
\newtheorem{lemma}[theor]{Lemma}
\newtheorem{propo}[theor]{Proposition}
\newtheorem{corol}[theor]{Corollary}

\theoremstyle{definition}
\newtheorem{defin}[theor]{Definition}

\newtheorem{notat}[theor]{Notation}
\newtheorem{quest}[theor]{Question}

\newtheorem{claim}[theor]{Claim}

\theoremstyle{remark}
\newtheorem{rem}[theor]{Remark}

\numberwithin{equation}{section}

\newcommand{\restrict}{\mathord{\upharpoonright}}

\renewcommand{\c}{\mathfrak{c}}
\renewcommand{\b}{\mathfrak{b}}
\renewcommand{\d}{{\mathfrak{d}}}
\newcommand{\s}{\mathfrak{s}}

\newcommand{\ttt}{{\mathfrak{t}}}
\renewcommand{\a}{{\mathfrak{a}}}
\newcommand{\ac}{{\mathfrak{a}}_{closed}}
\newcommand{\lsb}{\left[}
\newcommand{\rsb}{\right]}

\newcommand{\lc}{\left|}
\newcommand{\rc}{\right|}

\newcommand\ZFC{\mathrm{ZFC}}
\newcommand\FIN{\mathrm{FIN}}

\DeclareMathOperator{\spt}{suppt}

\DeclareMathOperator{\Lim}{Lim}
\renewcommand{\[}{\left[}
\renewcommand{\]}{\right]}
\newcommand{\Pset}{\mathcal{P}}

\newcommand{\BB}{\mathcal{B}}

\newcommand{\GG}{{\mathcal{G}}}

\newcommand{\cube}{{\lsb\omega\rsb}^{\omega}}

\newcommand{\I}{{\mathcal{I}}}

\newcommand{\J}{{\mathcal{J}}}

\newcommand{\FFF}{{\mathcal{F}}}
\newcommand{\FF}{{\mathbb{F}}}

\newcommand{\KK}{{\mathcal{K}}}

\newcommand{\pr}[2]{\langle #1, #2 \rangle}

\newcommand{\seq}[4]{\langle {#1}_{#2}: #2 #3 #4 \rangle}
\newcommand{\upp}[3]{\displaystyle\coprod_{#1 \in #2}{{#3}_{#1}}}
\newcommand{\upi}[3]{\displaystyle\coprod_{#1}{{#2}_{#3}}}

\DeclareMathOperator{\dom}{\bold{domain}}

\newcommand{\AND}{\text{ and }}

\newcommand{\SetOf}[2]{\left\{#1 \ \left| \ #2 \right.\right\}}

	\title[The almost disjointness invariant for products]{The almost disjointness invariant for products of ideals}
 \author[D. Raghavan]{Dilip Raghavan}
 \address{Department of Mathematics\\
  National University of Singapore\\
  Singapore 119076.}
 \email{dilip.raghavan@protonmail.com}
 \urladdr{https://dilip-raghavan.github.io/}
	\author[J. Stepr\={a}ns]{Juris Stepr\={a}ns}
	\address{Department of Mathematics\\
  York University\\
	4700 Keele Street\\
	Toronto, Ontario, Canada M3J 1P3.}
	\curraddr{}
	\email{steprans@yorku.ca}
\begin{document}
 \begin{abstract}
  The almost disjointness numbers associated to the quotients determined by the transfinite products of the ideal of finite sets are investigated.
  A $\mathord\mathrm{ZFC}$ lower bound involving the minimum of the classical almost disjointness and splitting numbers is proved for these characteristics. 
  En route, it is shown that the splitting numbers associated to these quotients are all equal to the classical splitting number.
  Finally, it is proved to be consistent that the almost disjointness numbers associated to these quotients are all equal to the second uncountable cardinal while the bounding number is the first uncountable cardinal.
  Several open problems are considered.
 \end{abstract}
\maketitle
\section{Introduction} \label{sec:intro}
Cardinal characteristics associated with definable ideals and their quotients have received considerable attention.
Two notable examples include the papers Brendle and Shelah~\cite{bs1} and Hern{\'a}ndez-Hern{\'a}ndez and Hru{\v{s}}{\'a}k~\cite{michaelstarinv}.
This paper will focus on the products of the ideal of finite subsets of $\omega$, including transfinite products, and the quotients they determine.
The main result will be a $\ZFC$ lower bound on the almost disjointness numbers associated to these quotients, and as a consequence of this $\ZFC$ lower bound, it will be shown to be consistent that $\b = {\aleph}_{1}$ while these almost disjointness numbers are all ${\aleph}_{2}$.
The next few definitions set the basic notation.
\begin{defin} \label{def:ideal}
 Given an infinite set $X$, $\I$ is said to be an \emph{ideal} on $X$ if $\I$ is a subset of $\Pset(X)$ such that the following conditions hold:
 \begin{enumerate}[series=ideal]
  \item \label{cond:nonprincipal}
  if $Y \subseteq X$ is finite, then $Y \in \I$;
  \item
  if $Y \in \I$ and $Z \subseteq Y$, then $Z \in \I$;
  \item
  if $Y \in \I$ and $Z \in \I$, then $Y \cup Z \in \I$;
  \item \label{cond:proper}
  $X \notin \I$.
 \end{enumerate}
Conditions (\ref{cond:nonprincipal}) and (\ref{cond:proper}) are often expressed as $\I$ is \emph{non-principal} and $\I$ is \emph{proper} respectively.
For an ideal $\I$ on $X$ and $A$ and $B$ subsets of $X$
define $A\equiv_{\mathcal I}B$ if and only if $A\triangle B\in \mathcal I$ and define
$A\subseteq_{\mathcal I}B$ if $A\setminus B\in \mathcal I$.
Then
$\Pset(X) \slash \I$ is a Boolean algebra of $\equiv_{\mathcal I}$ equivalence classes and ${\lsb Y \rsb}_{\I}$ denotes the $\equiv_{\mathcal I}$ equivalence class of $Y$, for any set $Y \subseteq X$.
A set $Y \in \Pset(X)$ is \emph{$\I$-positive} if $Y \notin \I$, equivalently ${\lsb Y \rsb}_{\I} > 0$.
This is often written as $Y \in {\I}^{+}$.
Observe that by (\ref{cond:nonprincipal}), every $\I$-positive set is infinite.

For $Y \in {\I}^{+}$, the \emph{restriction of $\I$ to $Y$}, denoted $\I\restrict Y$ is $\{Z \in \I: Z \subseteq Y\}$.
It is easy to see that $\I\restrict Y$ is an ideal on $Y$.
Attention will be restricted to ideals with the property that $\Pset(X) \slash \I$ is non-atomic, which is an easy consequence (as will be seen in Proposition~\ref{prop:nonmaximalnonatomic}) of the following additional condition:
\begin{enumerate}[resume=ideal]
 \item 
 for every $Y \subseteq X$, if $Y \notin \I$, then $\I \restrict Y$ is not a maximal ideal on $Y$.
\end{enumerate} 
 All ideals to be considered in this paper will enjoy this property.
\end{defin}
It is worth bearing in mind that an ideal $\I$ on a set $X$ is not maximal if there exist disjoint subsets $Y, Z \subseteq X$ with $Y, Z \notin \I$.
\begin{defin} \label{def:a}
 Let $\BB$ be a non-atomic Boolean algebra.
 An \emph{antichain in $\BB$} is a set $A \subseteq \BB$ such that $\forall a \in A\lsb a > 0 \rsb$ and $\forall a, a' \in A\lsb a \neq a' \implies a \wedge a' = 0 \rsb$.
 A \emph{maximal antichain in $\BB$} is a set $A \subseteq \BB$ which is an antichain in $\BB$ and which is not a proper subset of any antichain in $\BB$.
 Define 
 \begin{align*}
  {\a}_{\BB} = \min\left\{ \lc A \rc: A \subseteq \BB \ \text{is an infinite maximal antichain in} \ \BB \right\}.
 \end{align*}
\end{defin}
Note that ${\a}_{\BB}$ is well-defined for any non-atomic Boolean algebra $\BB$, for it is possible to find an infinite antichain below every $b \in \BB \setminus \{0\}$ in view of the lack of atoms.
The absence of atoms also ensures that the splitting number, define in Definition~\ref{def:s}, is well-defined for $\BB$.
\begin{defin}\label{def:s}
 $\FFF \subseteq \BB$ is called a \emph{splitting family in $\BB$} if for every $b \in \BB$ with $b > 0$, there exists $a \in \FFF$ with $b \wedge a > 0$ and $b \wedge (1-a) > 0$.
 Define
 \begin{align*}
  {\s}_{\BB} = \min\{\lc \FFF \rc: \FFF \subseteq \BB \ \text{and} \ \FFF \ \text{is a splitting family in} \ \BB \}
 \end{align*}
\end{defin}
\begin{propo} \label{prop:nonatomicsplit}
 If $\BB$ is a non-atomic Boolean algebra, then $\BB$ is a splitting family in $\BB$.
\end{propo}
Conversely, if $\BB$ has atoms, then there are no splitting families in $\BB$ as the atoms cannot be split, and hence ${\s}_{\BB}$ is not well-defined.
Condition (5) of Definition \ref{def:ideal} ensures that $\Pset(X) \slash \I$ will be non-atomic.
\begin{propo} \label{prop:nonmaximalnonatomic}
 Let $\I$ be an ideal on an infinite set $X$ satisfying condition (5) of Definition \ref{def:ideal}.
 Then $\Pset(X) \slash \I$ is non-atomic.
\end{propo}
\begin{proof}
 Consider $a \in \Pset(X) \slash \I$ with $a > 0$.
 So $a = {\[Y\]}_{\I}$ for some $Y \in \Pset(X)$ such that $Y \in {\I}^{+}$.
 By hypothesis $\I \restrict Y$ is not a maximal ideal on $Y$.
 Let $\J$ be an ideal on $Y$ so that $\I \restrict Y \subsetneq \J$.
 Fix $Z \in \J \setminus (\I \restrict Y)$.
 $Z \subseteq Y$ because $\J$ is an ideal on $Y$ and $Z \in \J$.
 As $\J$ is a proper ideal, $Y\setminus Z \notin \I\restrict Y$.
 Thus $Y\setminus Z \notin \I$ and $Z \notin \I$.
 Hence $0 < {\[Z\]}_{\I}, {\[Y\setminus Z\]}_{\I} \leq {\[Y\]}_{\I}$.
 Further, ${\[Y \setminus Z\]}_{\I} \wedge {\[Z\]}_{\I} = {\[(Y\setminus Z) \cap Z\]}_{\I} = 0$, showing that $a$ is not an atom.
\end{proof}
When $\BB$ is of the form $\Pset(X) \slash \I$, ${\a}_{\Pset(X)\slash \I}$ and ${\s}_{\Pset(X)\slash \I}$ will sometimes be rewritten as ${\a}_{\I}$ and ${\s}_{\I}$.
The Fubini square (see Definition~\ref{def:J}) of the ideal $\FIN$ of finite subsets of $\omega$ is usually denoted as $\FIN \times \FIN$ and it is most naturally viewed as an ideal on $\omega \times \omega$.
The quotient $\Pset(\omega \times \omega) \slash \left( \FIN \times \FIN \right)$ and its cardinal characteristics have been studied by several researchers.
A notable difference between $\Pset(\omega) \slash \FIN$ and $\Pset(\omega \times \omega) \slash \left( \FIN \times \FIN \right)$, first noticed by Szyma\'{n}ski and Zhou~\cite{szzh}, is that the tower number ${\ttt}_{\Pset(\omega \times \omega) \slash \left( \FIN \times \FIN \right)}$ is provably equal to ${\aleph}_{1}$ in $\ZFC$.

In an unpublished work~\cite{Brend}, Brendle investigated the almost disjointness number of the quotient $\Pset(\omega \times \omega) \slash \left( \FIN \times \FIN \right)$.
He observed that $\b \leq {\a}_{\Pset(\omega \times \omega) \slash \left( \FIN \times \FIN \right)} \leq \a$ and by using a template style iteration along the lines of Shelah~\cite{sh:700}, he was able to show the consistency of $\b = \d < {\a}_{\Pset(\omega \times \omega) \slash \left( \FIN \times \FIN \right)}$.
\begin{theor}[Brendle~\cite{Brend}]\label{thm:Brend}
 It is consistent that ${\aleph}_{2} = \b = \d < {\a}_{\Pset(\omega \times \omega) \slash \left( \FIN \times \FIN \right)}$.
\end{theor}
As template iterations will not produce models with ${\aleph}_{1} = \b$, Brendle~\cite{Brend} asked whether ${\aleph}_{1} = \b < {\a}_{\Pset(\omega \times \omega) \slash \left( \FIN \times \FIN \right)}$ is consistent.
Section \ref{sec:prod} of this paper will provide a positive answer to Brendle's question: it is consistent that ${\aleph}_{1} = \b$ and that ${\a}_{\Pset(\omega \times \omega) \slash \left( \FIN \times \FIN \right)} = {\aleph}_{2}$.
\section{Products of ideals} \label{sec:prod}
This section begins with some preliminary results about products of ideals in general.
We then focus on products that are supported on the ideal of finite subsets of the index set.
The results about transfinite products of $\FIN$ are then derived as corollaries.
It turns out that the splitting number of the quotients induced by such products is important to understanding the almost disjointness of these quotients.
\begin{defin} \label{def:upsidedown}
 For an indexed family $\seq{A}{x}{\in}{X}$, define
 \begin{align*}
 \displaystyle\coprod_{x \in X}{{A}_{x}} = \displaystyle\bigcup_{x \in X}{\left( \{x\} \times {A}_{x} \right)}.
 \end{align*}
\end{defin}
\begin{defin} \label{def:ax}
 For any sets $A$ and $x$, define $A(x) = \{y: \pr{x}{y}\in A\}$.
\end{defin}
The next proposition summarizes some basic attributes of the Definitions \ref{def:upsidedown} and \ref{def:ax}.
The proofs are straightforward applications of the definitions and are left to the reader.
\begin{propo} \label{prop:ax}
 The following properties hold:
 \begin{enumerate}
  \item
  $A \subseteq B \implies A(x) \subseteq B(x)$;
  \item
  $(A \cup B)(x) = A(x) \cup B(x)$;
  \item
  $(A \cap B)(x) = A(x) \cap B(x)$;
  \item
  $(A\setminus B)(x)=A(x)\setminus B(x)$;
  \item
  suppose $\seq{D}{x}{\in}{X}$ is an  indexed family of sets, $D = \upp{x}{X}{D}$, $A \subseteq X$, $\seq{E}{x}{\in}{A}$ is an  indexed family such that $\forall x \in A\[{E}_{x} \subseteq {D}_{x}\]$, and $E = \upp{x}{A}{E}$; then the following hold:
  \begin{enumerate}
   \item[(a)]
    $E \subseteq D$;
   \item[(b)]
   $\forall x \in A\[ {E}_{x} = E(x) \]$;
   \item[(c)]
   $\forall x \notin A \[E(x) = \emptyset\]$.
  \end{enumerate}
 \end{enumerate}
\end{propo}
\begin{defin} \label{def:J}
 Let $X$ be an infinite set and let $\seq{D}{x}{\in}{X}$ be an  indexed family of infinite sets.
 Suppose $\seq{\I}{x}{\in}{X}$ is an  indexed family so that ${\I}_{x}$ is an ideal on ${D}_{x}$, for all $x \in X$.
 Let $D = \upp{x}{X}{D}$.
 For any $A \subseteq D$, define
 \begin{align*}
  \spt(A) = \{x \in X: A(x) \notin {\I}_{x}\}.
 \end{align*}
 Given an ideal $\I$ on $X$, define
 \begin{align*}
 \upi{\I}{\I}{x} = \left\{ A \subseteq D: \spt(A) \in \I \right\}.
 \end{align*}
\end{defin}
\begin{lemma} \label{lem:Jisanideal}
 $\J = \upi{\I}{\I}{x}$ is an ideal on the infinite set $D$.
\end{lemma}
\begin{proof}
 As $X$ is infinite and as each ${D}_{x}$ is infinite, it is clear that $D$ is infinite.
 By definition $\J \subseteq \Pset(D)$.
 Suppose $A$ and $B$ satisfy $B \subseteq A$ and $A \in \J$.
 So $A \subseteq D$ and $\spt(A) \in \I$.
 As $B \subseteq A$, it is easy to see that $\spt(B) \subseteq \spt(A)$, whence $\spt(B) \in \I$.
 Therefore, $B \in \J$, as required for condition (2) of Definition \ref{def:ideal}.
 
 For condition (3) of Definition \ref{def:ideal}, fix $A, B \in \J$.
 Thus $A, B \subseteq D$ and $\spt(A) \cup \spt(B) \in \I$.
 Consider any $x \in \spt(A \cup B)$.
 Then $x \in X$ and, since ${\I}_{x}$ is an ideal on ${D}_{x}$, either $A(x) \notin {\I}_{x}$ or $B(x) \notin {\I}_{x}$.
 It follows that $x \in \spt(A) \cup \spt(B)$.
 So $\spt(A \cup B) \subseteq \spt(A) \cup \spt(B)$, which implies that $\spt(A \cup B) \in \I$.
 Therefore, $A\cup B \in \J$, as wanted.
 
 Next, note that by Proposition \ref{prop:ax}, $\forall x \in X \[{D}_{x} = D(x)\]$.
 Hence for all $x \in X$, ${\I}_{x}$ is a proper ideal on $D(x)$, and so $D(x) \notin {\I}_{x}$.
 Hence $\spt(D) = X\notin \I$.
 Therefore, $D\notin\J$.
 
 Finally suppose $E \subseteq D$ is finite.
 Then $\spt(E) = \emptyset$.
 To see this suppose otherwise that $x \in \spt(E)$.
 Then $x\in X$ and $E(x) \notin {\I}_{x}$.
 However, $E(x)$ is finite and $E(x) \subseteq D(x) = {D}_{x}$, which implies $E(x) \in {\I}_{x}$ because ${\I}_{x}$ is an ideal on ${D}_{x}$.
 This contradiction shows that $\spt(E) = \emptyset \in \I$.
 Therefore, $E \in \J$, completing the proof.
\end{proof}
Isomorphism between ideals is considered next.
This becomes necessary in order to show that the sequence of cardinal invariants to be defined does not depend on an arbitrary choice of bijections.
\begin{defin}\label{def:}
 Suppose $X$ and $Y$ are infinite sets, and $\I$ and $\J$ are ideals on $X$ and $Y$ respectively.
 We say \emph{$\I$ is isomorphic to $\J$} if there is a bijection $f: X \rightarrow Y$ such that
 \begin{align*}
  \J = \left\{ f'' A: A \in \I \right\}.
 \end{align*}
 $f$ is called \emph{an isomorphism from $\I$ to $\J$}.
\end{defin}
The following proposition lists two important properties of isomorphisms which will be useful in the proof of Lemma \ref{lem:iso}.
Their proofs are easy and left to the reader.
In (1) of Proposition \ref{prop:iso} and in the rest of the paper, ${\FIN}_{X}$ is the ideal of finite subsets of $X$.
\begin{propo}\label{prop:iso}
 Let $X, Y$ be infinite sets. Then
 \begin{enumerate}
  \item
   if $f: X \rightarrow Y$ is a bijection, then ${\FIN}_{Y} = \left\{f''A: A \in {\FIN}_{X} \right\}$;
  \item
   if $\I$ is an ideal on $X$, $\J$ is an ideal on $Y$, and $g: X \rightarrow Y$ is an isomorphism from $\I$ to $\J$, then for any $A \subseteq X$, if $A \notin \I$, then $g'' A \notin \J$.
 \end{enumerate}
\end{propo}
Proposition \ref{prop:iso} is used to prove the following lemma.
Its proof is a straightforward unraveling of the definitions.
Details are left to the reader.
\begin{lemma} \label{lem:iso}
 Let $X, Y$ be infinite sets and $\seq{D}{x}{\in}{X}$ and $\seq{E}{y}{\in}{Y}$ be indexed families of infinite sets. 
 Let $f: X \rightarrow Y$ be any bijection.
 Suppose $\seq{\I}{x}{\in}{X}$, $\seq{\J}{y}{\in}{Y}$, and $\seq{g}{x}{\in}{X}$ are indexed families so that ${\I}_{x}$ is an ideal on ${D}_{x}$, ${\J}_{y}$ is an ideal on ${E}_{y}$, and ${g}_{x}: {D}_{x} \rightarrow {E}_{f(x)}$ is an isomorphism from ${\I}_{x}$ to ${\J}_{f(x)}$, for all $x \in X$.
 Then if $\I$ is an ideal on $X$, $\J$ is an ideal on $Y$, and $f$ is an isomorphism from $\I$ to $\J$, then
 \begin{align*}
  \upi{\I}{\I}{x} \ \text{is isomorphic to} \ \upi{\J}{\J}{y}.
 \end{align*}
\end{lemma}
We define an ${\omega}_{1}$-sequence of countable sets and ideals on them by starting with ${\FIN}_{\omega}$ and iterating the operation of taking products that are supported on ${\FIN}_{\omega}$ through all the countable ordinals.
These ideals have been studied at least since the 1960s in the context of convergence of functions, for example see \cite{GG} and \cite{Kat}.
\begin{defin}\label{def:Dfinalpha}
 $\Lim({\omega}_{1}) = \left\{ \alpha < {\omega}_{1}: \alpha \ \text{is a limit ordinal} \right\}$.
 Let $\bar{s} = \seq{s}{\alpha}{\in}{\Lim({\omega}_{1})}$ be an  indexed family so that $\forall \alpha \in \Lim({\omega}_{1})\[{s}_{\alpha}: \omega \rightarrow \alpha \ \text{is a bijection}\]$.
 By induction on $\alpha < {\omega}_{1}$, define an infinite set $D(\bar{s}, \alpha)$ and an ideal $\FIN(\bar{s}, \alpha)$ on $D(\bar{s}, \alpha)$ as follows:
 \begin{enumerate}
  \item
  if $\alpha=0$, then $D(\bar{s}, \alpha)=\omega$ and $\FIN(\bar{s}, \alpha)={\FIN}_{\omega}$;
  \item
  when $\alpha=\xi+1$, and an infinite set $D(\bar{s}, \xi)$ as well as an ideal $\FIN(\bar{s}, \xi)$ on $D(\bar{s}, \xi)$ are given, then $D(\bar{s}, \alpha) = \displaystyle\coprod_{n\in\omega}{D(\bar{s}, \xi)}$ and $\FIN(\bar{s}, \alpha) = \displaystyle\coprod_{{\FIN}_{\omega}}{\FIN(\bar{s}, \xi)}$;
  by Lemma \ref{lem:Jisanideal}, $D(\bar{s}, \alpha)$ is an infinite set and $\FIN(\bar{s}, \alpha)$ is an ideal on $D(\bar{s}, \alpha)$; 
  \item
  when $\alpha$ is a limit ordinal, and for each $n \in \omega$, an infinite set $D(\bar{s}, {s}_{\alpha}(n))$ as well as an ideal $\FIN(\bar{s}, {s}_{\alpha}(n))$ on $D(\bar{s}, {s}_{\alpha}(n))$ are given, then $D(\bar{s}, \alpha) = \displaystyle\coprod_{n\in\omega}{D(\bar{s}, {s}_{\alpha}(n))}$ and $\FIN(\bar{s}, \alpha) = \displaystyle\coprod_{{\FIN}_{\omega}}{\FIN(\bar{s}, {s}_{\alpha}(n))}$;
  once again by Proposition \ref{lem:Jisanideal}, $D(\bar{s}, \alpha)$ is an infinite set and $\FIN(\bar{s}, \alpha)$ is an ideal on $D(\bar{s}, \alpha)$.
 \end{enumerate}
\end{defin}
\begin{lemma}\label{lem:noalpha}
 Let $\bar{s} = \seq{s}{\alpha}{\in}{\Lim({\omega}_{1})}$ and $\bar{t} = \seq{t}{\alpha}{\in}{\Lim({\omega}_{1})}$ be indexed families such that
 \begin{align*}
  \forall \alpha \in \Lim({\omega}_{1})\[{s}_{\alpha}: \omega \rightarrow \alpha \ \text{and} \ {t}_{\alpha}: \omega \rightarrow \alpha \ \text{are bijections}\].
 \end{align*}
 Then for each $\alpha < {\omega}_{1}$, $\FIN(\bar{s}, \alpha)$ is isomorphic to $\FIN(\bar{t}, \alpha)$.
\end{lemma}
\begin{proof}
 Proceed by induction on $\alpha < {\omega}_{1}$.
 If $\alpha=0$, then by definition, $D(\bar{s}, \alpha) = \omega=D(\bar{t}, \alpha)$ and $\FIN(\bar{s}, \alpha) = {\FIN}_{\omega} = \FIN(\bar{t}, \alpha)$.
 Hence the identity map on $\omega$ is an isomorphism.
 
 Suppose $\alpha=\xi+1$.
 By the induction hypothesis, there exists $g: D(\bar{s}, \xi) \rightarrow D(\bar{t}, \xi)$ which is an isomorphism from $\FIN(\bar{s}, \xi)$ to $\FIN(\bar{t}, \xi)$.
 Lemma \ref{lem:iso} is applicable with $\omega$ as $X$ and $Y$, $\left\langle D(\bar{s}, \xi): n \in \omega \right\rangle$ as $\seq{D}{x}{\in}{X}$, $\left\langle D(\bar{t}, \xi): n \in \omega \right\rangle$ as $\seq{E}{y}{\in}{Y}$, the identity map on $\omega$ as $f$, $\left\langle \FIN(\bar{s}, \xi): n \in \omega \right\rangle$ as $\seq{\I}{x}{\in}{X}$, $\left\langle \FIN(\bar{t}, \xi): n \in \omega \right\rangle$ as $\seq{\J}{y}{\in}{Y}$, $\left\langle g: n \in \omega \right\rangle$ as $\seq{g}{x}{\in}{X}$, and ${\FIN}_{\omega}$ as $\I$ and $\J$.
 And it yields the conclusion that $\FIN(\bar{s}, \xi+1)$ is isomorphic to $\FIN(\bar{t}, \xi+1)$.
 
 Now assume that $\alpha$ is a limit ordinal.
 $f: \omega \rightarrow \omega$ defined by $f(n) = {t}^{-1}_{\alpha}\left({s}_{\alpha}(n)\right)$ is a bijection.
 By item (1) of Proposition \ref{prop:iso}, $f$ is an isomorphism from ${\FIN}_{\omega}$ to ${\FIN}_{\omega}$.
 For any $n \in \omega$, by the induction hypothesis, there is ${g}_{n}: D(\bar{s}, {s}_{\alpha}(n)) \rightarrow D(\bar{t}, {s}_{\alpha}(n))$ which is an isomorphism from $\FIN(\bar{s}, {s}_{\alpha}(n))$ to $\FIN(\bar{t}, {s}_{\alpha}(n))$.
 Hence Lemma \ref{lem:iso} is applicable with $\omega$ as $X$ and $Y$, $\left\langle D(\bar{s}, {s}_{\alpha}(n)): n \in \omega \right\rangle$ as $\seq{D}{x}{\in}{X}$, $\left\langle D(\bar{t}, {t}_{\alpha}(n)): n \in \omega \right\rangle$ as $\seq{E}{y}{\in}{Y}$, $f$ as $f$, $\left\langle \FIN(\bar{s}, {s}_{\alpha}(n)): n \in \omega \right\rangle$ as $\seq{\I}{x}{\in}{X}$, $\left\langle \FIN(\bar{t}, {t}_{\alpha}(n)): n \in \omega \right\rangle$ as $\seq{\J}{y}{\in}{Y}$, $\seq{g}{n}{\in}{\omega}$ as $\seq{g}{x}{\in}{X}$, and ${\FIN}_{\omega}$ as $\I$ and $\J$.
 And it yields the conclusion that $\FIN(\bar{s}, \alpha)$ is isomorphic to $\FIN(\bar{t}, \alpha)$.
\end{proof}
Therefore the choice of $\bar{s}$ is inconsequential to the properties of the quotients $\Pset(D(\bar{s}, \alpha)) \slash \FIN(\bar{s}, \alpha)$.
The splitting numbers of these quotients will now be considered.
The first observation is that the splitting numbers are well-defined.
\begin{lemma}\label{lem:notmaximal1}
 Suppose $W$ is a countably infinite set.
 Let $\seq{D}{w}{\in}{W}$ be an  indexed family of infinite sets and let $\seq{\I}{w}{\in}{W}$ be an  indexed family such that for each $w \in W$, ${\I}_{w}$ is an ideal on ${D}_{w}$.
 Suppose $E=\upp{w}{W}{D}$ and $\J=\upi{{\FIN}_{W}}{\I}{w}$.
 Then for any $A \subseteq E$ with $A \notin \J$, $\J\restrict A$ is not maximal on $A$.
\end{lemma}
\begin{proof}
 As $A\notin \J$, $\spt(A) \subseteq W$ is infinite.
 Find disjoint infinite sets $X, Y$ with $X \cup Y = \spt(A)$.
 Let $B= \displaystyle\coprod_{w\in X}{A(w)}$ and $C=\displaystyle\coprod_{w\in Y}{A(w)}$.
 Then for any $w\in X$, $B(w)=A(w) \notin {\I}_{w}$ and for any $w \in Y$, $C(w)=A(w)\notin{\I}_{w}$.
 It is thus clear that  $X \subseteq \spt(B)$ and $Y \subseteq \spt(C)$.
 Therefore, $B, C \notin \J$, whence $B, C \notin \J\restrict A$.
 Furthermore, it is simple to check that $B \cap C = \emptyset$.
 Thus $C \subseteq A \setminus B$, and as $\J \restrict A$ is an ideal on $A$, it follows that $A\setminus B \notin \J\restrict A$.
 Since $B\subseteq A$, $\J \restrict A$ is an ideal on $A$, and $A \setminus B \notin \J\restrict A$, there is an ideal $\KK$ on $A$ such that $\left( \J\restrict A \right) \cup \{B\} \subseteq \KK$.
 $\J\restrict A \subsetneq \KK$ because $B \in \KK \setminus \left( \J \restrict A \right)$, showing that $\J\restrict A$ is not a maximal ideal on $A$.
\end{proof}
\begin{corol}\label{cor:salpha}
  Let $\bar{s} = \seq{s}{\alpha}{\in}{\Lim({\omega}_{1})}$ be so that $\forall \alpha \in \Lim({\omega}_{1})\[{s}_{\alpha}: \omega \rightarrow \alpha \ \text{is a bijection}\]$.
  The following hold:
  \begin{enumerate}
   \item
    for each $\alpha < {\omega}_{1}$ and for each $A \subseteq D(\bar{s}, \alpha)$, if $A \notin \FIN(\bar{s}, \alpha)$, then $\FIN(\bar{s}, \alpha)\restrict A$ is not maximal on $A$;
   \item
    for each $\alpha < {\omega}_{1}$ and for each $A \subseteq D(\bar{s}, \alpha)$, if $A \notin \FIN(\bar{s}, \alpha)$, then ${\s}_{\Pset(A) \slash \left( \FIN(\bar{s}, \alpha)\restrict A \right)}$ is well-defined.
  \end{enumerate}
\end{corol}
\begin{proof}
 It has already been established that each $D(\bar{s}, \alpha)$ is an infinite set and that $\FIN(\bar{s}, \alpha)$ is an ideal on $D(\bar{s}, \alpha)$.
 Item (1) is proved by induction on $\alpha$.
 If $\alpha = 0$, then $D(\bar{s}, \alpha) = \omega$ and $\FIN(\bar{s}, \alpha) = {\FIN}_{\omega}$.
 So if $A \subseteq D(\bar{s}, \alpha)$ and if $A \notin \FIN(\bar{s}, \alpha)$, then $A$ is an infinite subset of $\omega$ and ${\FIN}_{\omega} \restrict A = {\FIN}_{A}$, which is not maximal on $A$.
 If $\alpha > 0$, then $D(\bar{s}, \alpha)$ has the form $\upp{n}{\omega}{D}$ and $\FIN(\bar{s}, \alpha)$ has the form $\upi{{\FIN}_{\omega}}{\I}{n}$, where each ${D}_{n}$ is an infinite set and ${\I}_{n}$ is an ideal on ${D}_{n}$.
 Thus \ref{lem:notmaximal1} yields the conclusion of item (1).
 
 To prove item (2), fix some $\alpha < {\omega}_{1}$ and $A \subseteq D(\bar{s}, \alpha)$ with $A \notin \FIN(\bar{s}, \alpha)$.
 Then $\FIN(\bar{s}, \alpha) \restrict A$ is an ideal on $A$, and according to Propositions \ref{prop:nonatomicsplit} and \ref{prop:nonmaximalnonatomic}, it needs to be seen that for every $B \subseteq A$ with $B \notin \FIN(\bar{s}, \alpha) \restrict A$, $\left(\FIN(\bar{s}, \alpha) \restrict A \right) \restrict B$ is not maximal on $B$.
 This is however clear because $\left(\FIN(\bar{s}, \alpha) \restrict A \right) \restrict B = \FIN(\bar{s}, \alpha) \restrict B$, which is not maximal on $B$ by item (1).
\end{proof}
\begin{lemma} \label{lem:sissame}
 Suppose $X$ is a countably infinite set.
 Suppose $\seq{D}{x}{\in}{X}$ and $\seq{\I}{x}{\in}{X}$ are indexed families so that for all $x\in X$, ${D}_{x}$ is an infinite set and ${\I}_{x}$ is an ideal on ${D}_{x}$ with the property that for each $A \subseteq {D}_{x}$, if $A \notin {\I}_{x}$, then ${\I}_{x}\restrict A$ is not maximal on $A$.
 Let $D= \upp{x}{X}{D}$ and $\J = \upi{{\FIN}_{X}}{\I}{x}$.
 Fix $A \subseteq D$ with $A \notin \J$.
 Then the following hold:
 \begin{enumerate}
  \item
   ${\s}_{\Pset(A)\slash \left(\J\restrict A \right)} \leq \s$;
  \item
   suppose $\kappa$ is a cardinal such that $\kappa < \s$ and for each $x \in \spt(A)$, $\kappa < {\s}_{\Pset(A(x)) \slash \left( {\I}_{x} \restrict A(x) \right)}$.
   Then ${\s}_{\Pset(A) \slash \left( \J \restrict A \right)} > \kappa$.
 \end{enumerate}
\end{lemma}
\begin{proof}
 We prove (1) first.
 Define $Y=\spt(A)$.
 $Y$ is a countably infinite set.
 By the definition of $\s$, find a splitting family $\GG = \left\{ {\[ {Y}_{\alpha} \]}_{{\FIN}_{Y}}: \alpha < \s \right\}$ in $\Pset(Y) \slash {\FIN}_{Y}$.
 For each $\alpha < \s$, define ${B}_{\alpha} = \displaystyle\coprod_{x \in {Y}_{\alpha}}{A(x)}$.
 Then ${B}_{\alpha} \subseteq A$.
 Now suppose $B \subseteq A$ and $B \notin \J\restrict A$.
 Then $\spt(B)$ is an infinite subset of $\spt(A)=Y$.
 By the choice of $\GG$, find $\alpha < \s$ so that both ${Y}_{\alpha} \cap \spt(B)$ and $\left( Y \setminus {Y}_{\alpha} \right) \cap \spt(B)$ are infinite.
 It is necessary to check that ${B}_{\alpha} \cap B$ and $\left( A \setminus {B}_{\alpha} \right) \cap B$ don't belong to $\J \restrict A$.
 It suffices to see that they are not in $\J$.
 For this, it is enough to show that $\spt({B}_{\alpha} \cap B)$ and $\spt \left( \left( A \setminus {B}_{\alpha} \right) \cap B \right)$ are infinite.
 It is not difficult to verify that ${Y}_{\alpha} \cap \spt(B) \subseteq \spt({B}_{\alpha} \cap B)$ and that $\left( Y \setminus {Y}_{\alpha} \right) \cap \spt(B) \subseteq \spt \left( \left( A \setminus {B}_{\alpha} \right) \cap B \right)$, which shows that $\spt({B}_{\alpha} \cap B)$ and $\spt \left( \left( A \setminus {B}_{\alpha} \right) \cap B \right)$ are infinite.
 It now follows that $\left\{ {\[ {B}_{\alpha} \]}_{\J \restrict A}: \alpha < \s \right\}$ is a splitting family in $\Pset(A) \slash \left( \J \restrict A \right)$, which establishes (1).
 
 To prove (2), let $\{{b}_{\alpha}: \alpha < \kappa\} \subseteq \Pset(A) \slash \left( \J \restrict A \right)$ be given.
 Select ${B}_{\alpha} \subseteq A$ with ${\[ {B}_{\alpha} \]}_{\Pset(A) \slash \left( \J \restrict A \right)} = {b}_{\alpha}$.
 Define $Y = \spt(A)$.
 As ${B}_{\alpha}(x) \subseteq A(x)$, the hypothesis on $\kappa$ implies that $\left\{ {\[ {B}_{\alpha}(x) \]}_{\Pset(A(x)) \slash \left( {\I}_{x} \restrict A(x) \right)}: \alpha < \kappa \right\}$ is not a splitting family in $\Pset(A(x)) \slash \left( {\I}_{x} \restrict A(x) \right)$, for each $x \in Y$.
 Hence for each $x \in Y$, it is possible to find ${B}_{x} \subseteq A(x)$ with ${B}_{x} \notin {\I}_{x} \restrict A(x)$ and with the property that for every $\alpha < \kappa$, either ${B}_{x} \: {\subseteq}_{\left( {\I}_{x} \restrict A(x) \right)} \: {B}_{\alpha}(x)$ or ${B}_{x} \: {\subseteq}_{\left( {\I}_{x} \restrict A(x) \right)} \: A(x) \setminus {B}_{\alpha}(x)$.
 Now for each $\alpha < \kappa$, let ${Y}_{\alpha} = \left\{ x \in Y: {B}_{x} \: {\subseteq}_{\left( {\I}_{x} \restrict A(x) \right)} \: {B}_{\alpha}(x) \right\}$.
 As $Y$ is a countably infinite set and $\kappa < \s$, there is an infinite set $Z \subseteq Y$ so that for all $\alpha < \kappa$, either $Z \: {\subseteq}_{{\FIN}_{Y}} \: {Y}_{\alpha}$ or $Z \: {\subseteq}_{{\FIN}_{Y}} \: Y \setminus {Y}_{\alpha}$.
 Define $B = \upp{x}{Z}{B}$.
 It is clear $B \subseteq A$ and it is easy to verify that $Z \subseteq \spt(B)$.
 Since $Z$ is infinite, $B \notin \J$, whence $B \notin \J\restrict A$.
 Therefore $b = {\[ B \]}_{\Pset(A) \slash \left( \J\restrict A \right)} > 0$, and it will be shown that for each $\alpha < \kappa$, either $b \leq {b}_{\alpha}$ or $b \leq 1-{b}_{\alpha}$, which will show that $\{{b}_{\alpha}: \alpha < \kappa\}$ is not a splitting family in $\Pset(A) \slash \left( \J\restrict A \right)$.
 To see this, fix $\alpha < \kappa$.
 There are two cases to consider.
 
 Case 1: $Z \: {\subseteq}_{{\FIN}_{Y}} \: {Y}_{\alpha}$.
 Thus $F = Z \setminus {Y}_{\alpha} \in {\FIN}_{Y}$.
 In this case, $b \leq {b}_{\alpha}$ holds.
 In other words, $B \setminus {B}_{\alpha} \in \J \restrict A$.
 To show this, it is enough to show $B \setminus {B}_{\alpha} \in \J$, which is implied by showing $\spt\left( B \setminus {B}_{\alpha} \right)$ is finite, which in turn is implied by showing $\spt\left( B \setminus {B}_{\alpha} \right) \subseteq F$.
 Indeed if $x \in \spt \left( B \setminus {B}_{\alpha} \right)$, then $x \in X$ and $\left( B \setminus {B}_{\alpha} \right) (x) \notin {\I}_{x}$.
 In particular, $B(x) \neq \emptyset$, which implies $x \in Z$ because if $x$ were not in $Z$, then by the definition of $B$ as $\upp{x}{Z}{B}$, $B(x)$ would be empty, contradicting $\left( B \setminus {B}_{\alpha} \right) (x) \notin {\I}_{x}$.
 If $x$ were in ${Y}_{\alpha}$, then by definition of ${Y}_{\alpha}$, ${B}_{x} \setminus {B}_{\alpha}(x)$ would be in ${\I}_{x} \restrict A(x)$.
 However this would be a contradiction because $\left( B \setminus {B}_{\alpha} \right) (x) \subseteq {B}_{x} \setminus {B}_{\alpha}(x)$.
 Therefore, $x \in Z \setminus {Y}_{\alpha} = F$.
 This concludes Case 1.
 
 Case 2: $Z \: {\subseteq}_{{\FIN}_{Y}} \: \left( Y \setminus {Y}_{\alpha} \right)$.
 Thus $F = Z \cap {Y}_{\alpha}$ is finite.
 In this case $b \leq 1-{b}_{\alpha}$ holds.
 In other words, $B \cap {B}_{\alpha} \in \J\restrict A$.
 For this, it suffices to show $B \cap {B}_{\alpha} \in \J$, which is implied by showing $\spt \left( B \cap {B}_{\alpha} \right)$ is finite, and this in turn is implied by showing $\spt \left( B \cap {B}_{\alpha} \right) \subseteq F$.
 Indeed if $x \in \spt \left( B \cap {B}_{\alpha} \right)$, then $x \in X$ and $\left( B \cap {B}_{\alpha} \right) (x) \notin {\I}_{x}$, whence $B(x) \cap {B}_{\alpha}(x) \notin {\I}_{x}$.
 This immediately gives $x \in Z \cap {Y}_{\alpha} = F$, concluding the proof of Case 2 and of the lemma.
\end{proof}
\begin{corol} \label{cor:sissame}
 Suppose $\bar{s} = \seq{s}{\alpha}{\in}{\Lim({\omega}_{1})}$ is so that $\forall \alpha \in \Lim({\omega}_{1}) \[ {s}_{\alpha}: \omega \rightarrow \alpha \ \text{is a bijection} \]$.
 Then for any $A \subseteq D(\bar{s}, \alpha)$ with $A \notin \FIN(\bar{s}, \alpha)$, ${\s}_{\Pset(A) \slash \left( \FIN(\bar{s}, \alpha) \restrict A \right)} = \s$.
\end{corol}
\begin{proof}
 By the results established above, $\kappa = {\s}_{\Pset(A) \slash \left( \FIN(\bar{s}, \alpha) \restrict A \right)}$ is always well-defined for any relevant $A$.
 Now proceed by induction on $\alpha$.
 If $\alpha = 0$, then $D(\bar{s}, \alpha) = \omega$ and $\FIN(\bar{s}, \alpha) = {\FIN}_{\omega}$.
 So $A \subseteq \omega$ with $A$ infinite.
 Thus $\FIN(\bar{s}, \alpha) \restrict A = \left( {\FIN}_{\omega} \right) \restrict A = {\FIN}_{A}$.
 Therefore $\Pset(A) \slash \left( \FIN(\bar{s}, \alpha) \restrict A \right) = \Pset(A) \slash {\FIN}_{A}$, which is isomorphic to $\Pset(\omega) \slash {\FIN}_{\omega}$.
 So ${\s}_{\Pset(A) \slash \left( \FIN(\bar{s}, \alpha) \restrict A \right)} = \s$.
 If $\alpha > 0$, then $D(\bar{s}, \alpha) = \upp{n}{\omega}{D}$ and $\FIN(\bar{s}, \alpha) = \upi{{\FIN}_{\omega}}{\I}{n}$, where each ${D}_{n}$ has the form $D(\bar{s}, \xi)$ and ${\I}_{n}$ has the form $\FIN(\bar{s}, \xi)$ for some $\xi < \alpha$.
 In particular, each ${D}_{n}$ is an infinite set and ${\I}_{n}$ is an ideal on ${D}_{n}$ with the property that for any $C \subseteq {D}_{n}$, if $C \notin {\I}_{n}$, then ${\I}_{n} \restrict C$ is not maximal on $C$.
 Hence $\kappa \leq \s$ by (1) of Lemma \ref{lem:sissame}.
 On the other hand, $\kappa$ cannot be strictly less than $\s$.
 For if $\kappa < \s$, then by the induction hypothesis, for each $n \in \spt(A)$, $\kappa < \s = {\s}_{\Pset(A(n)) \slash \left( {\I}_{n} \restrict A(n) \right)}$.
 And so by (2) of Lemma \ref{lem:sissame}, there can be no splitting family of size $\kappa$ in $\Pset(A) \slash \left( \FIN(\bar{s}, \alpha) \restrict A \right)$, contradicting the definition of $\kappa$.
 Therefore $\kappa = \s$.
\end{proof}
\begin{quest} \label{q:others}
 What are the possible values for $\mathfrak s_{\mathcal I}$ relative to other cardinal invariants if $\mathcal I$ is an $F_\sigma$ ideal?
 What of the specific case for the summable ideal? 
\end{quest}
Corollary \ref{cor:sissame} says that the splitting number is not a new cardinal invariant for the ideals from Definition \ref{def:Dfinalpha} or for any of their restrictions.
Their almost disjointness numbers are examined next.
\begin{defin} \label{def:aalpha}
 For each $\alpha \in {\omega}_{1}$, define ${\a}_{\alpha} = {\a}_{\Pset(D(\bar{s}, \alpha)) \slash \FIN(\bar{s}, \alpha)}$, where $\bar{s} = \seq{s}{\alpha}{\in}{\Lim({\omega}_{1})}$ is any sequence so that $\forall \alpha \in \Lim({\omega}_{1})\[{s}_{\alpha}: \omega \rightarrow \alpha \ \text{is a bijection}\]$.
 By Lemma \ref{lem:noalpha}, the choice of $\bar{s}$ is immaterial.
\end{defin}
\begin{propo}\label{Bpoig}
 If  $\mathcal J_n$ are ideals on the infinite sets $Y_n$ for $n\in \omega$, then $\mathfrak a_{\coprod_{{\FIN}_{\omega}} \mathcal J_n}\geq \mathfrak b$.
\end{propo}
\begin{proof}
 Suppose that $\mathcal A$ is a family such that:
 \begin{itemize}
  \item $|\mathcal A|<\mathfrak b$
  \item $\mathcal A\subseteq (\coprod_{{\FIN}_{\omega}} \mathcal J_n)^+$
  \item if $A$ and $B$ are distinct elements of $\mathcal A$ then $A\cap B \in \coprod_{{\FIN}_{\omega}} \mathcal J_n$.
 \end{itemize}
 Let $\{A_n\}_{n\in \omega}$ be distinct elements of $\mathcal A$. It is easy to find $A_n^*\subseteq A_n$ such that
 $A_n^*\notin \coprod_{{\FIN}_{\omega}} \mathcal J_n$ and $A^*_n\cap A^*_m=\varnothing$ for distinct $n$ and $m$. 
 For each $A\in \mathcal A\setminus \{A_n\}_{n\in \omega}$ define $F_A:\omega \to \omega$ such that
 $A(k)\cap A_n^*(k)\in \mathcal J_k$ for all $k\geq F_A(n)$.
 
 There is then
 $F:\omega \to \omega$ such that $F\geq^* F_A$ for all $A\in \mathcal A\setminus \{A_n\}_{n\in \omega}$.
 Find a function $F': \omega \rightarrow \omega$ such that for each $n \in \omega$, $F'(n) > F(n)$ and ${A}^{\ast}_{n}(F'(n)) \notin {\J}_{F'(n)}$.
 Define $V=\coprod_{n\in \omega}A^*_n(F'(n))$. It is routine to check that $V\cap A\in \coprod_{{\FIN}_{\omega}} \mathcal J_n$ for each $A\in \mathcal A$ showing that $\mathcal A$ is not maximal.
\end{proof}
\begin{corol}\label{Kkjsgd}
 $\mathfrak a_\alpha\geq \mathfrak b$ for all $\alpha \in \omega_1$.
\end{corol}
\begin{propo} \label{prop:aupper}
 If ${\I}_{n}$ are ideals on the infinite sets ${D}_{n}$ for $n \in \omega$, then ${\a}_{\J} \leq \a$, where $\J$ is the ideal $\upi{{\FIN}_{\omega}}{\I}{n}$ on $D = \upp{n}{\omega}{D}$.
\end{propo}
\begin{proof}
 Let $\{{Y}_{\alpha}: \alpha < \a \}$ be a m.a.d.\@ family on $\omega$.
 Define ${A}_{\alpha} = \upp{n}{{Y}_{\alpha}}{D}$.
 Then ${A}_{\alpha} \subseteq D$ and ${A}_{\alpha} \notin \J$ because ${Y}_{\alpha} \subseteq \spt({A}_{\alpha})$.
 If $\alpha < \beta < \a$ and $n \in \spt({A}_{\alpha} \cap {A}_{\beta})$, then ${A}_{\alpha}(n) \cap {A}_{\beta}(n) = \left( {A}_{\alpha} \cap {A}_{\beta} \right)(n) \notin {\I}_{n}$, whence $n \in {Y}_{\alpha} \cap {Y}_{\beta}$.
 Thus $\spt({A}_{\alpha} \cap {A}_{\beta}) \subseteq {Y}_{\alpha} \cap {Y}_{\beta}$, which is finite, and so ${A}_{\alpha} \cap {A}_{\beta} \in \J$.
 Finally, suppose that $A \subseteq D$ with $A \notin \J$.
 Then $\spt(A)$ is an infinite subset of $\omega$.
 Find $\alpha < \a$ so that $\spt(A) \cap {Y}_{\alpha}$ is infinite.
 If $n \in \spt(A) \cap {Y}_{\alpha}$, then $\left( A \cap {A}_{\alpha} \right)(n) = A(n) \cap {A}_{\alpha}(n) = A(n) \cap {D}_{n} = A(n) \notin {\I}_{n}$, whence $n \in \spt\left( A \cap {A}_{\alpha} \right)$.
 Thus $\spt(A) \cap {Y}_{\alpha} \subseteq \spt\left( A \cap {A}_{\alpha} \right)$, and so $A \cap {A}_{\alpha} \notin \J$.
 This shows that $\left\{ {\[{A}_{\alpha}\]}_{\Pset(D) \slash \J}: \alpha < \a \right\}$ is an infinite maximal almost disjoint family in $\Pset(D) \slash \J$.
\end{proof}
\begin{corol} \label{cor:aupper}
 For each $\alpha < {\omega}_{1}$, ${\a}_{\alpha} \leq \a$.
\end{corol}
Of course, ${\a}_{0}$ is the classical invariant $\a$.
By Corollaries \ref{Kkjsgd} and \ref{cor:aupper}, the cardinals ${\a}_{\alpha}$ stand sandwiched between $\b$ and $\a$.
It is unknown at present whether any of the ${\a}_{\alpha}$ can be distinguished from each other.
\begin{quest} \label{q:manya}
 Is it consistent to have ${\a}_{\alpha} < {\a}_{\beta}$ for some $\alpha, \beta < {\omega}_{1}$?
 For each $n \geq 1$, is it consistent to have $\b < {\a}_{n} < \dotsb < \a$?
\end{quest}
The invariant ${\a}_{1}$ was investigated by Brendle~\cite{Brend}.
Brendle considered both the questions of whether or not $\mathfrak a_1<\mathfrak a$ is consistent and weather or not $\b < {\a}_{1}$ is consistent.
He used a template style iteration similar to the one from Shelah~\cite{sh:700} to produce a model where ${\aleph}_{2} = \b < \a = {\a}_{1} = {\aleph}_{3}$.
Brendle~\cite{Brend} asked whether ${\aleph}_{1} = \b < {\a}_{1}$ is consistent.
The question of whether or not $\mathfrak a_1<\mathfrak a$ is consistent was again implicitly raised in \cite{AsgSch} and explicitly by T\"{o}rnquist at the Fields Set Theory meeting in May 2019.
\begin{quest}
What are the possible values for $\mathfrak a_{\mathcal I}$ relative to other cardinal invariants if $\mathcal I$ is an $F_\sigma$ ideal? What of the specific case for the summable ideal? 
\end{quest}
The next theorem, which is the main result of this paper, provides a $\ZFC$ lower bound for ${\a}_{\alpha}$ in terms of $\a$ and $\s$.
It will shed some light on Question \ref{q:manya} by constraining possible models of ${\a}_{\alpha} < \a$.
\begin{theor} \label{thm:main}
 Let $X$ be a countably infinite set and let $\seq{D}{x}{\in}{X}$ be an  indexed family of infinite sets.
 Suppose $\seq{\I}{x}{\in}{X}$ is an  indexed family such that ${\I}_{x}$ is an ideal on ${D}_{x}$ with the property that for every $A \subseteq {D}_{x}$, if $A \notin {\I}_{x}$, then ${\I}_{x} \restrict A$ is not maximal on $A$.
 Let $\kappa$ be an infinite cardinal.
 Assume that for each $x \in X$ and for every $A \subseteq {D}_{x}$ with $A \notin {\I}_{x}$, $\kappa < {\s}_{\Pset(A) \slash \left( {\I}_{x} \restrict A \right)}$ and that $\kappa < \a$.
 Let $D = \upp{x}{X}{D}$ and $\J = \upi{{\FIN}_{X}}{\I}{x}$.
 Then no $\left( \Pset(D) \slash \J \right)$-almost disjoint sequence $\seq{a}{\alpha}{<}{\kappa}$ is maximal in $\Pset(D) \slash \J$.
\end{theor}
\begin{proof}
 Choose ${A}_{\alpha} \in {\left( \Pset(D) \slash \J \right)}^{+}$ so that ${\[ {A}_{\alpha} \]}_{\Pset(D) \slash \J} = {a}_{\alpha}$.
 Let ${X}_{\alpha} = \spt({A}_{\alpha}) \subseteq X$ and note that ${X}_{\alpha} \notin {\FIN}_{X}$.
 Thus for each $\alpha < \kappa$, $\lc {X}_{\alpha} \rc = {\aleph}_{0}$.
 Find a family $\{{Y}_{n}: n \in \omega\}$ so that ${Y}_{n}$ is infinite and ${Y}_{n} \subseteq {X}_{n}$, for all $n \in \omega$, and $\forall n < m < \omega\[ {Y}_{n} \cap {Y}_{m} = \emptyset \]$.
 Define $Y = {\bigcup}_{n \in \omega}{{Y}_{n}}$.
 Fix $x \in Y$ and let ${n}_{x}$ be the unique $n \in \omega$ so that $x \in {Y}_{n}$.
 Then by definition, $x \in X$, ${A}_{{n}_{x}}(x) \subseteq {D}_{x}$, and ${A}_{{n}_{x}}(x) \notin {\I}_{x}$.
 The assumption that $\kappa < {\s}_{\Pset\left( {A}_{{n}_{x}}(x) \right) \slash \left( {\I}_{x} \restrict {A}_{{n}_{x}}(x) \right)}$ implies that $\left\{ {\[{A}_{{n}_{x}}(x) \cap {A}_{\alpha}(x) \]}_{\Pset\left( {A}_{{n}_{x}}(x) \right) \slash \left( {\I}_{x} \restrict {A}_{{n}_{x}}(x) \right)}: \alpha < \kappa \right\}$ is not a splitting family in $\Pset\left( {A}_{{n}_{x}}(x) \right) \slash \left( {\I}_{x} \restrict {A}_{{n}_{x}}(x) \right)$.
 So there is ${B}_{x} \subseteq {A}_{{n}_{x}}(x)$ such that ${B}_{x} \notin {\I}_{x}$ and for any $\alpha < \kappa$, either ${B}_{x} \: {\subseteq}_{{\I}_{x}} \: {A}_{{n}_{x}}(x) \cap {A}_{\alpha}(x)$ or ${B}_{x} \: {\subseteq}_{{\I}_{x}} \: {A}_{{n}_{x}}(x) \setminus {A}_{\alpha}(x)$.
 Now for each $\alpha < \kappa$, define ${Z}_{\alpha} = \left\{ x \in Y: {B}_{x} \: {\subseteq}_{{\I}_{x}} \: {A}_{{n}_{x}}(x) \cap {A}_{\alpha}(x) \right\}$.
 Although the following claim is simple, it plays an important role.
 \begin{claim} \label{claim:main1}
  For each $n \in \omega$, ${Y}_{n} \subseteq {Z}_{n}$.
 \end{claim}
 \begin{proof}
  $x \in {Y}_{n} \implies {n}_{x} = n$, and by choice of ${B}_{x}$, ${B}_{x} \subseteq {A}_{n}(x) = {A}_{n_x}(x) \cap {A}_{n}(x)$.
  Hence $x \in Y$ and ${B}_{x} \: {\subseteq}_{{\I}_{x}} \: {A}_{{n}_{x}}(x) \cap {A}_{n}(x)$, whence $x \in {Z}_{n}$ by definition.
 \end{proof}
 \begin{claim} \label{claim:main2}
  For all $\alpha < \beta < \kappa$, ${Z}_{\alpha} \cap {Z}_{\beta}$ is finite.
 \end{claim}
 \begin{proof}
  Suppose for a contradiction that ${Z}_{\alpha} \cap {Z}_{\beta}$ is infinite.
  Consider $x \in {Z}_{\alpha} \cap {Z}_{\beta}$.
  Then $x \in X$, ${B}_{x} \: {\subseteq}_{{\I}_{x}} \: {A}_{{n}_{x}}(x) \cap {A}_{\alpha}(x)$ and ${B}_{x} \: {\subseteq}_{{\I}_{x}} \: {A}_{{n}_{x}}(x) \cap {A}_{\beta}(x)$.
  Hence ${B}_{x} \: {\subseteq}_{{\I}_{x}} \: {A}_{\alpha}(x) \cap {A}_{\beta}(x) \subseteq {A}_{\alpha}(x) \subseteq {D}_{x}$.
  As ${\I}_{x}$ is an ideal on ${D}_{x}$ and ${B}_{x} \notin {\I}_{x}$, ${A}_{\alpha}(x) \cap {A}_{\beta}(x) \notin {\I}_{x}$.
  So $\left( {A}_{\alpha} \cap {A}_{\beta} \right) (x) \notin {\I}_{x}$.
  Thus $x \in \spt \left( {A}_{\alpha} \cap {A}_{\beta} \right)$.
  It has been shown that ${Z}_{\alpha} \cap {Z}_{\beta} \subseteq \spt \left( {A}_{\alpha} \cap {A}_{\beta} \right)$, whence $\spt \left( {A}_{\alpha} \cap {A}_{\beta} \right)$ is infinite.
  So ${A}_{\alpha} \cap {A}_{\beta} \notin \J$.
  However this implies ${a}_{\alpha} \wedge {a}_{\beta} = {\[ {A}_{\alpha} \]}_{\Pset(D) \slash \J} \wedge {\[ {A}_{\beta} \]}_{\Pset(D) \slash \J} = {\[ {A}_{\alpha} \cap {A}_{\beta} \]}_{\Pset(D) \slash \J} > 0$, contradicting the almost disjointness of $\seq{a}{\xi}{<}{\kappa}$.
 \end{proof}
 Let $T = \{ \alpha < \kappa: {Z}_{\alpha} \ \text{is infinite} \}$.
 By Claim \ref{claim:main1} and by the fact that each ${Y}_{n}$ is infinite, $\omega \subseteq T$.
 Therefore by Claim \ref{claim:main2}, $\FF = \{ {Z}_{\alpha}: \alpha \in T \}$ is an infinite almost disjoint family of infinite subsets of $Y$ with $\lc \FF \rc \leq \lc T \rc \leq \kappa < \a$.
 As $Y$ is a countably infinite set, fix $Z \subseteq Y$ infinite with $\lc Z \cap {Z}_{\alpha} \rc < {\aleph}_{0}$ for all $\alpha < \kappa$.
 Define $B = \upp{x}{Z}{B}$.
 Now for each $x \in Z$, ${B}_{x} \notin {\I}_{x}$.
 Hence $Z \subseteq \spt(B)$.
 As $Z$ is infinite, $\spt(B)$ is infinite as well and so $B \notin \J$.
 \begin{claim} \label{claim:main3}
  ${\[ B \]}_{\Pset(D) \slash \J}$ is $\Pset(D) \slash \J$-almost disjoint to $\{{a}_{\alpha}: \alpha < \kappa\}$.
 \end{claim}
 \begin{proof}
  As $B \subseteq D$ and $B \notin \J$, $b = {\[ B \]}_{\Pset(D) \slash \J} > 0$.
  Fix $\alpha < \kappa$.
  Then $b \wedge {a}_{\alpha} = {\[ B \cap {A}_{\alpha} \]}_{\Pset(D) \slash \J}$.
  Consider any $x \in \spt \left( B \cap {A}_{\alpha} \right)$.
  Then $x \in X$ and $B(x) \cap {A}_{\alpha}(x) \notin {\I}_{x}$.
  In particular, $x \in Z$ because otherwise $B(x) = \emptyset \in {\I}_{x}$.
  Thus $x \in Y$.
  Now assume for a contradiction that $x \notin {Z}_{\alpha}$.
  Then ${B}_{x} \: {\subseteq}_{{\I}_{x}} \: {A}_{{n}_{x}}(x) \setminus {A}_{\alpha}(x)$.
  However as ${B}_{x} \subseteq {A}_{{n}_{x}}(x) \subseteq {D}_{x}$, it follows that $B(x) \cap {A}_{\alpha}(x) \in {\I}_{x}$, which is a contradiction.
  This contradiction shows that $x \in Z \cap {Z}_{\alpha}$.
  Thus it has been shown that $\spt \left( B \cap {A}_{\alpha} \right) \subseteq Z \cap {Z}_{\alpha}$.
  As $\alpha < \kappa$, $Z \cap {Z}_{\alpha}$ is finite.
  So $\spt \left( B \cap {A}_{\alpha} \right)$ is finite.
  Therefore $B \cap {A}_{\alpha} \in \J$, whence $b \wedge {a}_{\alpha} = 0$.
 \end{proof}
 Claim \ref{claim:main3} shows that $\seq{a}{\alpha}{<}{\kappa}$ is not maximal.
\end{proof}
\begin{corol} \label{cor:main1}
 For each $\alpha < {\omega}_{1}$, ${\a}_{\alpha} \geq \min\{\a, \s\}$.
\end{corol}
\begin{proof}
 Let $\kappa = {\a}_{\alpha}$, which is always an infinite cardinal.
 When $\alpha = 0$, $\kappa = \a \geq \min\{\a, \s\}$.
 When $\alpha > 0$, $D(\bar{s}, \alpha)$ has the form $\upp{n}{\omega}{D}$ and $\FIN(\bar{s}, \alpha)$ has the form $\upi{{\FIN}_{\omega}}{\I}{n}$, where each ${D}_{n}$ is an infinite set and ${\I}_{n}$ is an ideal on ${D}_{n}$.
 Furthermore, each ${D}_{n}$ has the form $D(\bar{s}, \xi)$ and ${\I}_{n}$ has the form $\FIN(\bar{s}, \xi)$.
 By Corollary \ref{cor:salpha}, for any $A \subseteq {D}_{n}$, if $A \notin {\I}_{n}$, then ${\I}_{n} \restrict A$ is not maximal on $A$.
 Now assume for a contradiction that $\kappa < \min\{\a, \s\}$.
 Then by Corollary \ref{cor:sissame}, for any $A \subseteq {D}_{n}$ with $A \notin {\I}_{n}$, $\kappa < \s = {\s}_{\Pset(A) \slash \left( {\I}_{n} \restrict A \right)}$.
 Therefore the hypotheses of Theorem \ref{thm:main} are satisfied and it implies that there are no maximal almost disjoint families of size $\kappa$ in $\Pset(D(\bar{s}, \alpha)) \slash \FIN(\bar{s}, \alpha)$, contradicting the definition of $\kappa$.
\end{proof}
The next theorem, which easily follows from Theorem \ref{thm:main} and from the work of Shelah in \cite{MR763901}, answers Brendle's question about the consistency of ${\aleph}_{1} = \b < {\a}_{1} = {\aleph}_{2}$.
\begin{theor} \label{label:baleph1}
 It is consistent that $\b = {\aleph}_{1}$ and for each $\alpha < {\omega}_{1}$, ${\a}_{\alpha} = {\aleph}_{2}$.
\end{theor}
\begin{proof}
 In \cite{MR763901}, Shelah produced a model in which $\b = {\aleph}_{1}$ and $\a = \s = \c = {\aleph}_{2}$.
 So Theorem \ref{thm:main} says that the necessary configuration holds in Shelah's model.
\end{proof}
In \cite{mob2}, Brendle found a c.c.c.\@ forcing closely related to Shelah's creature forcing from \cite{MR763901} to produce models where $\b$ is small while $\a$ and $\s$ are both larger.
In fact, Brendle~\cite{mob2} showed that for any regular $\kappa$, there is a model with $\b = \kappa < {\kappa}^{+} = \a = \s$.
Brendle's methods from \cite{mob2} yield the following consistency result.
\begin{corol}\label{label:largea}
 Let $\kappa$ be a regular uncountable cardinal.
 It is consistent with $\ZFC$ that $\kappa = \b$ and for each $\alpha < {\omega}_{1}$, ${\a}_{\alpha} = \s = \c = {\kappa}^{+}$.
\end{corol}
Brendle and Khomskii~\cite{perfectmad} introduced the cardinal $\ac$, which is the least $\kappa$ such that there are $\kappa$ many closed subsets of $\cube$ whose union is a m.a.d.\@ family in $\Pset(\omega) \slash {\FIN}_{\omega}$.
Brendle and Khomskii showed in \cite{perfectmad} that $\ac = {\aleph}_{1} < {\aleph}_{2} = \b$ holds in the Hechler model.
Since $\b \leq {\a}_{\alpha}$, this shows the consistency of $\ac < {\a}_{\alpha}$ for any $\alpha < {\omega}_{1}$.
However the following seems unclear.
\begin{quest} \label{q:ac}
 Is there some $\alpha < {\omega}_{1}$ for which ${\a}_{\alpha} < \ac$ is consistent?
\end{quest}
Brendle and Raghavan~\cite{bsa} adapted the arguments of Shelah~\cite{MR763901} and Brendle~\cite{mob2} to produce models of $\b < \ac$.
Their models show the following.
\begin{corol} \label{cor:ac}
 Let $\kappa$ be a regular uncountable cardinal.
 It is consistent with $\ZFC$ that $\kappa = \b$ and for each $\alpha < {\omega}_{1}$, $\ac = {\a}_{\alpha} = \s = \c = {\kappa}^{+}$.
\end{corol}
None of the presently available techniques seem to produce a model where $\b = {\aleph}_{1}$ and $\a$ is larger than ${\aleph}_{2}$.
So the following basic question seems to be open.
\begin{quest} \label{q:smallblargea}
 Is it consistent to have $\b = {\aleph}_{1}$ and $\a > {\aleph}_{2}$?
\end{quest}
For $\alpha > 0$, Theorem \ref{thm:main} contains some information about the possible techniques that could be used to exhibit a model of ${\a}_{\alpha} < \a$.
$\b < \a$ must hold in any such model and three essentially different techniques are presently known for getting models with $\b < \a$.
The first class of techniques derive from Shelah's method in \cite{MR763901} and its c.c.c.\@ version discovered by Brendle~\cite{mob2}.
All of these techniques increase $\s$ in addition to $\a$ and they tend to produce models where $\a = \s = \c$, and so by Theorem \ref{thm:main}, they are unsuitable for ${\a}_{\alpha} < \a$.
This class of techniques remains the only currently available one for getting ${\aleph}_{1} = \b < \a$.
Hence it will be necessary to solve the following open problem in order to show the consistency of ${\aleph}_{1} = {\a}_{\alpha} < \a$ for some $\alpha < {\omega}_{1}$.
\begin{quest}[Brendle and Raghavan]\label{q:BR}
 Is it consistent to have ${\aleph}_{1} = \b = \s < \a$?
\end{quest}
It is consistent to have ${\aleph}_{2} = \b = \s < \a$.
In fact, Shelah~\cite{sh:700} showed the consistency of $\d < \a$ and he invented two different techniques for this.
The first method involves a measurable cardinal and the ultrapower of a well chosen forcing notion.
The second is the method for iteration along a template.
However neither of these techniques will produce a model of ${\a}_{\alpha} < \a$ because the argument which shows that $\a$ is increased by these forcings will also apply to ${\a}_{\alpha}$.
It should be noted that unlike the case for $\a$, Raghavan and Shelah~\cite{borelmadness} showed that $\ac = {\aleph}_{1}$ if $\d = {\aleph}_{1}$.
\end{document}